\documentclass[leqno,12pt]{amsart}
\usepackage[width=0.8\paperwidth,height=0.85\paperheight,twoside=false,includehead,includefoot]{geometry}
\usepackage{amssymb}
\usepackage{graphicx}
\allowdisplaybreaks[1]

\theoremstyle{plain} 
\newtheorem{theorem}{\indent\sc Theorem}[section]
\newtheorem{lemma}[theorem]{\indent\sc Lemma}

\newtheorem{remark}[theorem]{\indent\sc Remark}
\theoremstyle{definition}
\newtheorem{definition}[theorem]{\indent\sc Definition}
\newtheorem{example}[theorem]{\indent\sc Example}

\newcommand{\A}{\mathcal{A}}

\newcommand{\Q}{\mathbb{Q}}
\newcommand{\Z}{\mathbb{Z}}
\newcommand{\R}{\mathbb{R}}

\def\k{\textrm{\textup{\textmd{\textbf{k}}}}}
\def\l{\textrm{\textup{\textmd{\textbf{l}}}}}
\def\e{\textrm{\textup{\textmd{\textbf{e}}}}}
\font\sevency=wncyr7 \def\sh{\,\hbox{\sevency X}}

\font\sevency=wncyr7 \def\sh{\,\hbox{\sevency X}}

\def\wc{\circle{4}}
\def\bc{\circle*{4}}

\def\boxp{\text{\,\fbox{1+}\,}}
\def\boxc{\text{\,\fbox{1,}\,}}
\def\boxo{\text{\,\fbox{1}\,}}

\begin{document}
\title[On a generalization of restricted sum formula]{On a generalization of restricted sum formula for multiple zeta values and finite multiple zeta values}

\author[H. Murahara]{Hideki Murahara}
\address{
Nakamura Gakuen University\endgraf 
5-7-1 Befu Jonan-ku Fukuoka-shi, Fukuoka, 814-0198\endgraf 
Japan
} 
\email{hmurahara@nakamura-u.ac.jp}

\author[T. Murakami]{Takuya Murakami}
\address{
Graduate School of Mathematics\endgraf 
Kyushu University\endgraf 
744 Motooka Fukuoka-shi Fukuoka, 819-0395\endgraf 
Japan
}
\email{tak\_mrkm@icloud.com}

\begin{abstract}
 We prove a new linear relation for multiple zeta values.
 This is a natural generalization of the restricted sum formula proved by Eie, Liaw and Ong.
 We also present an analogous result for finite multiple zeta values. 
\end{abstract}
\maketitle

\section{Main results}
\subsection{Main result for multiple zeta values}
 For $k_1,\dots,k_r\in \Z_{\ge1}$ with $k_1 \ge 2$, the multiple zeta values (MZVs) and the multiple zeta-star values (MZSVs) are defined respectively by 
\begin{align*}
 \zeta(k_1,\dots, k_r):=\sum_{n_1>\cdots >n_r \ge 1} \frac {1}{n_1^{k_1}\cdots n_r^{k_r}}
\end{align*} 
and
\begin{align*}
 \zeta^\star (k_1,\dots, k_r):=\sum_{n_1\ge \cdots \ge n_r \ge 1} \frac {1}{n_1^{k_1}\cdots n_r^{k_r}}. 
\end{align*} 
They are both generalizations of the Riemann zeta values $\zeta(k)$ at positive integers. 

For an index $\k=(k_1,\dots,k_r)$, we call $|\k|:=k_1+\cdots+k_r$ the weight and $r$ the depth.
We write 
$\zeta^+ (k_1,\dots,k_r):=\zeta (k_1+1,k_2,\dots,k_r)$. 
For two indices $\k$ and $\l$, we denote by $\k+\l$ the index obtained by componentwise addition, and always assume implicitly the depths of both $\k$ and $\l$ are equal.
We also write $\l\geq 0$ if every component of $\l$ is a non-negative integer.
Our first main result is the following:

\begin{theorem} \label{main1}
For $(k_1,\dots,k_r)\in\mathbb{Z}_{\ge1}^r, t \in\mathbb{Z}_{\ge0}$, we have 
 \begin{align*}
  \begin{split}
   &\sum_{\substack{ m_1+\cdots +m_r=r+t \\ m_i\ge1 \, (1\le i\le r) }} \,
   \sum_{\substack{ |\bold{a}_{m_i}|=k_i+m_i-1 \\ (1\le i\le r)} }
   \zeta^+ (\bold{a}_{m_1},\dots,\bold{a}_{m_r}) \\
   &=\sum_{l=0}^{t}
   \sum_{\substack{ m_1+\cdots +m_{r-1}=t-l \\ m_i\ge0 \, (1\le i\le r-1) }} \,
   \sum_{ \substack{|\bold{e}|=l \\ \e\geq 0}} \,\,
   \zeta^+ ( (k_1,\{1\}^{m_1}, \dots ,k_{r-1}, \{1\}^{m_{r-1}} ,k_r) + \bold{e} ). 
  \end {split}
 \end{align*} 
Here and hereafter, each $\bold{a}_{m_i}$ denotes an $m_i$-tuple of positive integers.
 When $r=1$, we understand the R.H.S. as $\zeta^+(k_1+t)$.
\end{theorem}

\begin{remark}
Theorem \ref{main1} is equivalent to the derivation relation which was obtained by K.\ Ihara, M.\ Kaneko and D.\ Zagier \cite{ihara_kaneko_zagier_2006}.  
This equivalence will be explained in Section 3.
\end{remark}

\begin{remark}
We can deduce the sum formula
\[
\sum_{\substack{s_1+\dots+s_{u}=k\\s_1\geq2,s_i\geq1(2\leq i \leq u)}}\zeta (s_1,\dots,s_{u})= \zeta(k)
\]
from Theorem \ref{main1} by taking $r=1, k_1=k-u $ and $t=u-1$ for any positive integers $k$ and $u$ with $k-u\geq 1$. 

\end{remark}

\begin{example}
For $(k_1,k_2)=(1,2)$, $t=1$, we have 
\[
  2\zeta(2, 1, 2) + \zeta(2, 2, 1)=\zeta(2, 3) + \zeta(3, 2) + \zeta(2, 1, 2).
 \]
\end{example}

Theorem \ref{main1} is also equivalent to the following Theorem \ref{main1-2}. 
\begin{theorem} \label{main1-2}
For $(k_1,\dots,k_r)\in\mathbb{Z}_{\ge1}^r, s, t\in\mathbb{Z}_{\ge0}$, we have 
 \begin{align*}
  &\sum_{\substack{ m_1+\cdots +m_r=r+t \\ m_i\ge1 \, (1\le i\le r) }} \,
  \sum_{\substack{ |\bold{a}_{m_i}|=k_i+m_i-1 \\ (1\le i\le r)} }
  \zeta^+ (\bold{a}_{m_1},\dots,\bold{a}_{m_r},\{1\}^{s}) \\
  &=\sum_{l=0}^{t}
  \sum_{\substack{ m_1+\cdots +m_{r-1}=t-l \\ m_i\ge0 \, (1\le i\le r-1) }} \,
  \sum_{\substack{ |\bold{e}|=l \\ \e\geq 0}} \,\,
  \zeta^+ ( (k_1,\{1\}^{m_1},\dots ,k_{r-1},\{1\}^{m_{r-1}} ,k_r, \{1\}^{s}) + \bold{e} ). 
 \end{align*} 
  When $r=1$, we understand the R.H.S. as 
  $ \sum_{\substack{ |\bold{e}|=t \\ \e\geq 0}} \,\,
  \zeta^+ ( (k_1, \{1\}^{s}) + \bold{e} )$.
\end{theorem}

This is a generalization of the restricted sum formula obtained by M.\ Eie, W--C.\ Liaw and Y.\ L.\ Ong \cite{eie_liaw_ong_2009}. 
The case $r=1$ gives the original formula. 
The proof of Theorem \ref{main1-2} will be given in Section 2. 
Here, we prove the equivalence of Theorem \ref{main1} and Theorem \ref{main1-2}.
\begin{proof}[\indent\sc Proof of the equivalence of Theorem \ref{main1} and Theorem \ref{main1-2}]
It is clear that Theorem \ref{main1-2} implies Theorem \ref{main1} by setting $s=0$.
So, we prove that Theorem \ref{main1} implies Theorem \ref{main1-2}.
Write $G(\k,s,t)$ (resp. $H(\k,s,t)$) for the left-hand side (resp. the right-hand side) of Theorem \ref{main1-2} and let $F(\k,s,t):=G(\k,s,t)-H(\k,s,t)$.
We prove $F(\k,s,t)=0$ for  $\k\in\mathbb{Z}_{\ge1}^r, s, t\in\mathbb{Z}_{\ge0}$ by induction on $s$.
If $s=0$, then $F(\k,0,t)=0$ by Theorem \ref{main1}.
We assume $F(\k,s,t)=0$ for some $s\in\mathbb{Z}_{\ge0}$ and show $F(\k,s+1,t)=0$.
\begin{align*}
 G((\k,1),s,t)
 =&\sum_{\substack{ m_1+\dots +m_{r+1}=r+t+1 \\ m_i\ge1 \, (1\le i\le r+1) } }
 \sum_{\substack{ |\bold{a}_{m_i}|=k_i+m_i-1 \\ (1\le i\le r)} }
  \zeta^+ (\bold{a}_{m_1},\dots,\bold{a}_{m_r},\{1\}^{m_{r+1}+s})\\
 =&\sum^{t+1}_{m_{r+1}=1}
  \sum_{\substack{ m_1+\dots+m_r=r+t-m_{r+1}+1 \\ m_i\ge1 \, (1\le i\le r) } }
  \sum_{\substack{ |\bold{a}_{m_i}|=k_i+m_i-1 \\ (1\le i\le r)} }
  \zeta^+ (\bold{a}_{m_1},\dots,\bold{a}_{m_r},\{1\}^{m_{r+1}+s})\\
 =&\sum^{t+1}_{m_{r+1}=1}G(\k,s+m_{r+1}, t-m_{r+1}+1)\\
 =&\sum^t_{u=0}G(\k,s+t-u+1,u),
\end{align*}
\begin{align*}
 H((\k,1),s,t)
 =&\sum^t_{l=0} 
 \sum_{\substack{ m_1+\dots+m_r=t-l \\ m_i\ge0 \, (1\le i\le r) } } 
 \sum_{\substack{ |\bold{e}|=l \\ \e\geq 0}} \,\,
  \zeta^+ ( (k_1,\{1\}^{m_1},\dots ,\{1\}^{m_{r-1}} ,k_r, \{1\}^{m_r+s+1}) + \bold{e} ) \\
 =&\sum^t_{l=0}\sum^{t-l}_{m_r=0}
 \sum_{\substack{ m_1+\dots+m_{r-1}=t-l-m_r \\ m_i\ge0 \, (1\le i\le r-1) } }
 \sum_{\substack{ |\bold{e}|=l \\ \e\geq 0}} \,\,
  \zeta^+ ( (k_1,\{1\}^{m_1},\dots ,\{1\}^{m_{r-1}} ,k_r, \{1\}^{m_r+s+1}) + \bold{e} ) \\
 =&\sum^{t}_{m_r=0}
 \sum^{t-m_r}_{l=0}
 \sum_{\substack{ m_1+\dots+m_{r-1}=t-l-m_r \\ m_i\ge0 \, (1\le i\le r-1) } }
 \sum_{\substack{ |\bold{e}|=l \\ \e\geq 0}} \,\,
  \zeta^+ ( (k_1,\{1\}^{m_1},\dots ,\{1\}^{m_{r-1}} ,k_r, \{1\}^{m_r+s+1}) + \bold{e} ) \\
 =&\sum^{t}_{m_r=0}
 H(\k, s+m_r+1, t-m_r)\\
 =&\sum^t_{u=0} H(\k,s+t-u+1,u).
\end{align*}
Therefore, we have
\[
F((\k,1),s,t)=\sum^{t}_{u=0}F(\k,s+t-u+1,u).
\]
By replacing $s$ to $s+1$ and $t$ to $t-1$, we have
\[
F((\k,1),s+1,t-1)=\sum^{t-1}_{u=0}F(\k,s+t-u+1,u).
\]
Take subtraction of the previous two equations, we have
\[
F(\k,s+1,t)=F((\k,1),s,t)-F((\k,1),s+1,t-1).
\]
By applying this equation repeatedly and $F(\k,s,0)=0$ for arbitrary index $\k$ and $s\in\mathbb{Z}_{\ge0}$, we obtain the following equation.
\[
F(\k,s+1,t)=\sum^{t}_{t'=1}(-1)^{t'-1}F((\k,\{1\}^{t'}),s,t-t'+1).
\]
Therefore, we obtain the desired result.
\end{proof}

\subsection{Main result for finite multiple zeta values}
There are two types of finite multiple zeta value (FMZV) being studied variously;
$\A$-finite multiple zeta(-star) values ($\A$-FMZ(S)Vs) and symmetric multiple zeta(-star) values (SMZ(S)Vs). 

We consider the collection of truncated sums $\zeta_{p}(k_1, \dots, k_r)=\sum_{p>n_1>\cdots>n_r\ge1}\frac{1}{n_1^{k_1}\cdots n_r^{k_r}}$ modulo all primes $p$ in the quotient ring $\A=(\prod_{p}\Z/p\Z)/(\bigoplus_{p}\Z/p\Z)$, which is a $\Q$-algebra. 
Elements of $\A$ are represented by $(a_p)_p$, where $a_p\in\Z/p\Z$, and two elements $(a_p)_p$ and $(b_p)_p$ are identified if and only if $a_p=b_p$ for all but finitely many primes $p$.
For $k_1, \dots ,k_r \in \Z_{\ge1}$, the $\A$-FMZVs and the $\A$-FMZSVs are defined by
\begin{align*}
 \zeta_{\mathcal{A}} (k_1, \dots, k_r)
 &:=\biggl( \sum_{p>n_1>\cdots >n_r\ge1}
 \frac{1}{n_1^{k_1}\cdots n_r^{k_r}} \bmod{p} \biggr)_{p}, \\ 
 \zeta_{\mathcal{A}}^{\star} (k_1, \dots, k_r)&:=\biggl( \sum_{p>n_1\ge\cdots\ge n_r\ge1}
 \frac{1}{n_1^{k_1}\cdots n_r^{k_r}} \bmod{p} \biggr)_{p}. 
\end{align*}
The SMZ(S)Vs was first introduced by Kaneko and Zagier \cite{kaneko_2015, kaneko_zagier_2017}. 
For $k_1,\dots ,k_r \in \Z_{\ge1}$, we let  
\begin{align*}
 \zeta_{\mathcal{S}}^{\ast} (k_1, \dots, k_r)
 &:=\sum_{i=0}^{r}(-1)^{k_1+\cdots +k_i}\zeta^{\ast}(k_i, \dots , k_1) \zeta^{\ast}(k_{i+1}, \dots,k_r). 
\end{align*}
Here, the symbol $\zeta^{\ast}$ on the right-hand side stands for the regularized value coming from harmonic regularization,
i.e.,  a real value obtained by taking constant terms of harmonic regularization as explained in   \cite{ihara_kaneko_zagier_2006}. 
In the sum, we understand $\zeta^{\ast}(\emptyset )=1$.
Let $\mathcal {Z}_{\R}$ be the $\Q$-vector subspace of $\R$ spanned by $1$ and all MZVs, which is a $\Q$-algebra. 
Then, the SMZVs is defined as an element 
in the quotient ring $\mathcal {Z}_{\R}/(\zeta (2))$ by 
\[
 \zeta_{\mathcal{S}} (k_1, \dots,k_r):=\zeta_{\mathcal{S}}^{\ast} (k_1, \dots,k_r) \bmod \zeta(2). 
\]
For $k_1, \dots,k_r \in \Z_{\ge1}$, we also define the SMZSVs in $\mathcal {Z}_{\R}/(\zeta (2))$ by  
\[ 
 \zeta_{\mathcal{S}}^{\star} (k_1, \dots,k_r)
 :=\sum_{ \substack{\square \textrm{ is either a comma ``," } \\ \textrm{ or a plus ``+"}} } 
 \zeta_{\mathcal{S}}^{\ast}(k_1 \square \cdots \square k_r) \bmod  \zeta(2). 
\]   

Denoting $\mathcal {Z}_{\A}$ by the $\Q$-vector subspace of $\A$ spanned by $1$ and all $\A$-FMZVs, Kaneko and Zagier conjecture that there is an isomorphism between $\mathcal {Z}_{\A}$ and $\mathcal {Z}_{\R}/\zeta(2)$ as $\Q$-algebras such that $\zeta_{\A}(k_1, \dots , k_r)$ and $\zeta_{\mathcal{S}}(k_1, \dots , k_r)$ correspond with each other. 
(For more details, see \cite{kaneko_2015, kaneko_zagier_2017}.) 
In the following, the letter ``$\mathcal{F}$'' stands either for ``$\mathcal{A}$'' or ``$\mathcal{S}$''.
Now, we state our second main result:

\begin{theorem} \label{main2}
For $(k_1,\dots,k_r)\in\mathbb{Z}_{\ge1}^r, t \in\mathbb{Z}_{\ge0}$, we have 
 \begin{align*}
  &\sum_{\substack{ m_1+\cdots +m_r=r+t \\ m_i\ge1 \, (1\le i\le r) }} \,
  \sum_{\substack{ |\bold{a}_{m_i}|=k_i+m_i-1 \\ (1\le i\le r)} }
  \zeta_{\mathcal{F}} (\bold{a}_{m_1},\dots,\bold{a}_{m_r}) \\
  &=\sum_{l=0}^{t}
  \sum_{\substack{ m_1+\cdots +m_{r-1}=t-l \\ m_i\ge0 \, (1\le i\le r-1) }} \,
  \sum_{\substack{ |\bold{e}|=l \\ \e\geq 0}} \,\,
  \zeta_{\mathcal{F}} ( (k_1,\{1\}^{m_1},\dots ,k_{r-1}, \{1\}^{m_{r-1}} ,k_r) + \bold{e} ). 
 \end{align*} 
  When $r=1$, we understand the R.H.S. as $\zeta_{\mathcal{F}}(k_1+t)$.
\end{theorem}

\begin{remark}
We can also obtain the FMZVs version of the restricted sum fomula by replacing $\zeta^+$ with $\zeta_{\mathcal{F}}$ in Theorem \ref{main1-2}. 
\end{remark}

\section{Proof of Theorem \ref{main1-2}}

\subsection{Integral Series Identity}
A 2-poset is a pair $(X,\delta_X)$, where $X=(X,\leq)$ is a finite partially ordered set and $\delta_X$ is a label map from $X$ to $\{0,1\}$. 
A 2-poset $(X,\delta_X)$ is called admissible if $\delta_X(x)=0$ for all maximal elements $x\in X$ and $\delta_X(x)=1$ for all minimal elements $x\in X$.

A 2-poset $(X,\delta_X)$ is depicted as a Hasse diagram in which an element $x$ with $\delta(x)=0$ (resp. $\delta(x)=1$) is represented by $\circ$ (resp. $\bullet$). For example, the diagram
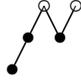
\[
\begin{picture}(30,30)
	\put(0,0)\bc
	\put(0,0){\line(1,2){6}}
	\put(6,12)\bc
	\put(6,12){\line(1,2){6}}
	\put(12,24)\wc
	\put(12,24){\line(1,-2){6}}
	\put(18,12)\bc
	\put(18,12){\line(1,2){6}}
	\put(24,24)\wc
\end{picture}
\]
represents the 2-poset $X=\{x_1,x_2,x_3,x_4,x_5\}$ with order $x_1<x_2<x_3>x_4<x_5$ and label $(\delta_X(x_1),\dots,\delta_X(x_5))=(1,1,0,1,0)$. This 2-poset is admissible.

For an admissible 2-poset $X$, we define the associated integral
\[
I(X):=\int_{\Delta_X}\prod_{x\in X}\omega_{\delta_{X}(x)}(t_x),
\]
where
\[
\Delta_X:=\{(t_x)_x\in[0,1]^X | t_x<t_y \textrm{ if } x<y\}
\]
and
\[
\omega_0(t):=\frac{dt}{t}, \ \ \omega_1(t):=\frac{dt}{1-t}.
\]
For example, 
\[
I\left(
\begin{picture}(35,20)(-5,10)
	\put(0,0)\bc
	\put(0,0){\line(1,2){6}}
	\put(6,12)\bc
	\put(6,12){\line(1,2){6}}
	\put(12,24)\wc
	\put(12,24){\line(1,-2){6}}
	\put(18,12)\bc
	\put(18,12){\line(1,2){6}}
	\put(24,24)\wc
\end{picture}
\right)=\int_{t_1<t_2<t_3>t_4<t_5}\frac{dt_1}{1-t_1}\frac{dt_2}{1-t_2}\frac{dt_3}{t_3}\frac{dt_4}{1-t_4}\frac{dt_5}{t_5}.
\]

For indices $\k=(k_1,\dots,k_r)$ and $\l=(l_1,\dots,l_s)$, we define $\mu(\k,\l)$ as a 2-poset corresponds to the following diagram.
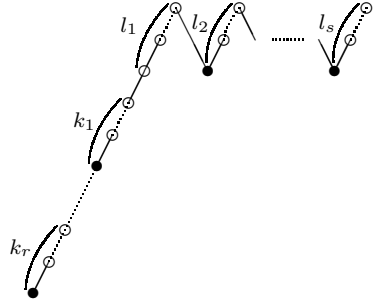
\[
\begin{picture}(86,110)
	\put(0,0)\bc
	\put(0,0){\line(1,2){6}}
	\put(6,12)\wc
	\qbezier[8](6,12)(9,18)(12,24)
%
	\put(12,24)\wc
	\qbezier[16](12,24)(18,36)(24,48)
	\qbezier(-3,1.5)(-2,14)(9,25.5)
	\put(-9,15){\tiny $k_r$}
	\put(24,48)\bc
	\put(24,48){\line(1,2){6}}
	\put(30,60)\wc
	\qbezier[8](30,60)(33,66)(36,72)
	\put(36,72)\wc
	\qbezier(21,49.5)(22,62)(33,73.5)
	\put(15,63){\tiny $k_1$}
	\put(36,72){\line(1,2){6}}
	\put(42,84)\wc
	\put(42,84){\line(1,2){6}}
	\put(48,96)\wc
	\qbezier[8](48,96)(51,102)(54,108)
	\put(54,108)\wc
	\qbezier(39,85.5)(40,98)(51,109.5)
	\put(33,99){\tiny $l_1$}
	\put(54,108){\line(1,-2){12}}
	\put(66,84)\bc
	\put(66,84){\line(1,2){6}}
	\put(72,96)\wc
	\qbezier[8](72,96)(73,102)(78,108)
	\put(78,108)\wc
	\qbezier(65.5,87.5)(65,99)(75,109.5)
	\put(60,100){\tiny $l_2$}
	\put(78,108){\line(1,-2){6}}
	\qbezier[8](90.5,96)(96,96)(102,96)
	\put(108,96){\line(1,-2){6}}
	\put(114,84)\bc
	\put(114,84){\line(1,2){6}}
	\put(120,96)\wc
	\qbezier[8](120,96)(123,102)(126,108)
	\put(126,108)\wc
	\qbezier(113.5,87.5)(113,99)(123,109.5)
	\put(108,100){\tiny $l_s$}
\end{picture}
\]

For an index $\k=(k_1,\dots,k_r)$, let $\k^\star$ be the formal sum of $2^{r-1}$ indices of the form $(k_1\square\cdots\square k_r)$, where each $\square$ is replaced by $``,"$ or $``+"$.

We also define the $\Q$-bilinear ``circled harmonic product'' $\circledast$ by 
\[
 (k_1,\dots,k_r)\circledast(l_1,\dots,l_s):=(k_1+l_1,(k_2,\dots,k_r)*(l_2,\dots,l_s)),
\]
where product ``$*$'' is the haronic product defined inductively by
\begin{align*}
&\qquad\qquad\varnothing*\k=\k*\varnothing=\k,
\\
&\k*\l=(k_1,\k'*\l)+(l_1,\k*\l')+(k_1+l_1,\k'*\l')
\end{align*}
for any indices $\k=(k_1,\k')$ and $\l=(l_1,\l')$.

Kaneko and Yamamoto proved the following formula for MZVs.
\begin{theorem}[Kaneko--Yamamoto\ \cite{kaneko_yamamoto_2016}] \label{integral_series_identity}
For any non-empty indices \k\ and \l, we have
\[
 \zeta(\mu(\k,\l))=\zeta(\k\circledast\l^\star).
\]
\end{theorem}

\subsection{Proof of Theorem \ref{main1-2}}

For $\k=(k_1,\dots,k_r,\{1\}^s)$ and $\l=(\{1\}^{t+1})$, we have
\begin{align*}
\zeta(\mu(\k,\l))=&I\left(
\begin{picture}(86,66)(-7,52)
	\put(0,0)\bc
	\qbezier[8](0,0)(3,6)(6,12)
	\put(6,12)\bc
	\qbezier(-3,1.5)(-3,9)(3,13.5)
	\put(-6,8){\tiny $s$}
	\put(6,12){\line(1,2){6}}
	\put(12,24)\bc
	\put(12,24){\line(1,2){6}}
	\put(18,36)\wc
	\qbezier[8](18,36)(21,42)(24,48)
	\put(24,48)\wc
	\qbezier(9,25.5)(10,38)(21,49.5)
	\put(3,39){\tiny $k_r$}
	\qbezier[8](24,48)(27,54)(30,60)
	\qbezier[8](30,60)(33,66)(36,72)
	\put(36,72)\bc
	\put(36,72){\line(1,2){6}}
	\put(42,84)\wc
	\qbezier[8](42,84)(45,90)(48,96)
	\put(48,96)\wc
	\qbezier(33,73.5)(34,86)(45,97.5)
	\put(27,87){\tiny $k_1$}
	\put(48,96){\line(1,2){6}}
	\put(54,108)\wc
	\put(54,108){\line(1,-2){6}}
	\put(60,96)\bc
	\qbezier[8](60,96)(63,90)(66,84)
	\qbezier[8](66,84)(69,78)(72,72)
	\put(72,72)\bc
	\qbezier(75,73.5)(74,86)(63,97.5)
	\put(75,87){\tiny $t$}
\end{picture}
\right)=\sum_{\substack{m_1+\dots+m_r+j=r+t\\(m_i\geq 1,j\geq0)}}I\left(
\begin{picture}(90,50)(-5,50)
	\put(40,95)\wc
	\put(30,90)\wc
	\put(30,80)\wc
	\put(50,90)\bc
	\put(50,80)\bc
	\put(40,75)\bc
	\put(40,95){\line(-2,-1){10}}
	\put(40,95){\line(2,-1){10}}
	\qbezier[6](30,90)(30,85)(30,80)
	\qbezier[6](50,90)(50,85)(50,80)
	\qbezier(27,90)(24,85)(27,80)		\put(0,82.5){\tiny $k_1-1$}
	\qbezier(53,90)(56,85)(53,80)		\put(56,82.5){\tiny $m_1-1$}
	\put(30,80){\line(2,-1){20}}
	\put(50,80){\line(-2,-1){20}}
	\put(30,70)\wc
	\put(30,60)\wc
	\put(50,70)\bc
	\put(50,60)\bc
	\put(40,55)\bc
	\qbezier[6](30,70)(30,65)(30,60)
	\qbezier[6](50,70)(50,65)(50,60)
	\qbezier(27,70)(24,65)(27,60)		\put(0,62.5){\tiny $k_2-1$}
	\qbezier(53,70)(56,65)(53,60)		\put(56,62.5){\tiny $m_2-1$}
	\put(30,60){\line(2,-1){10}}
	\put(50,60){\line(-2,-1){10}}
	\qbezier[6](40,55)(40,50)(40,45)
	\put(40,45)\bc
	\put(30,40)\wc
	\put(30,30)\wc
	\put(50,40)\bc
	\put(50,30)\bc
	\put(40,25)\bc
	\put(40,45){\line(-2,-1){10}}
	\put(40,45){\line(2,-1){10}}
	\qbezier[6](30,40)(30,35)(30,30)
	\qbezier[6](50,40)(50,35)(50,30)
	\qbezier(27,40)(24,35)(27,30)		\put(0,32.5){\tiny $k_r-1$}
	\qbezier(53,40)(56,35)(53,30)		\put(56,32.5){\tiny $m_r-1$}
	\put(30,30){\line(2,-1){10}}
	\put(50,30){\line(-2,-1){10}}
	\put(40,25){\line(1,-1){10}}
	\put(40,25){\line(-1,-1){10}}
	\put(30,15)\bc
	\put(50,15)\bc
	\put(20,5)\bc
	\put(60,5)\bc
	\qbezier[8](30,15)(25,10)(20,5)
	\qbezier[8](50,15)(55,10)(60,5)
	\qbezier(28,17)(20,15)(18,7)	\put(17,15){\tiny $s$}
	\qbezier(52,17)(60,15)(62,7)	\put(62,15){\tiny $j$}
\end{picture}
\right)\\
=&\sum^t_{j=0}\binom{s+j}{s}
\sum_{\substack{m_1+\dots+m_r=r+t-j \\ m_i\ge1 \, (1\le i\le r) }}
 \sum_{\substack{ |\bold{a}_{m_i}|=k_i+m_i-1 \\ (1\le i\le r)} }
  \zeta^+ (\bold{a}_{m_1},\dots,\bold{a}_{m_r},\{1\}^{s+j}).
\end{align*}

In general, for $\k'=(k_1,\dots,k_{r+s})$ and $\l=(\{1\}^{t+1})$, we have
\begin{align*}
 \zeta(\k'\circledast\l^\star)=\sum^t_{l=0}
 \sum_{\substack{m_1+\dots+m_{r+s}=r+s+t-l \\ m_i\ge1 \, (1\le i\le r) } }
 \sum_{\substack{|\e|=l \\ \e\geq 0}} \zeta^+((k_1,\{1\}^{m_1-1},\dots,k_{r+s},\{1\}^{m_{r+s}-1})+\e)
\end{align*}
because the index $(\{1\}^{t+1})^\star$ is equal to the formal sum of all indices of weight $t+1$.
Now, we put $k_{r+1}=\dots=k_{r+s}=1$ here. Then, the index 
$(k_1,\{1\}^{m_1-1},\dots,k_{r+s},\{1\}^{m_{r+s}-1})$ on the right becomes $(k_1,\{1\}^{m_1-1},\dots,k_{r-1},\{1\}^{m_{r-1}-1},k_r,\{1\}^{u-1})$ with $u=m_{r}+\dots+m_{r+s}$. For a fixed $u$, the number of $(s+1)$-tuple
$(m_r,\dots,m_{r+s})$ giving $u=m_r+\cdots+m_{r+s}$ is $\binom{u-1}{s}$.
%
Thus,
\begin{align*}
 \zeta(\k\circledast\l^\star)
=\sum^t_{l=0}\sum^{s+t-l+1}_{u=s+j}\binom{u-1}s
 \sum_{\substack{m_1+\dots+m_{r-1}=r+t+s-l-u\\m_i\geq1}}
 \sum_{\substack{|\e|=l \\ \e\geq 0}}\zeta^+((k_1,\{1\}^{m_1-1},\dots,k_{r},\{1\}^{u-1})+\e).
\end{align*}
By writing $u=s+j+1$,
\begin{align*}
 \zeta(\k\circledast\l^\star)
=\sum^t_{l=0}\sum^{t-l}_{j=0}\binom{s+j}s
 \sum_{\substack{m_1+\dots+m_{r-1}=r+t-l-j-1\\m_i\geq1}}
 \sum_{\substack{|\e|=l \\ \e\geq 0}}\zeta^+((k_1,\{1\}^{m_1-1},\dots,k_{r},\{1\}^{s+j})+\e)\\
 =\sum^{t}_{j=0}\binom{s+j}s\sum^{t-j}_{l=0}
 \sum_{\substack{m_1+\dots+m_{r-1}=r+t-l-j-1\\m_i\geq1}}
 \sum_{\substack{|\e|=l \\ \e\geq 0}}\zeta^+((k_1,\{1\}^{m_1-1},\dots,k_{r},\{1\}^{s+j})+\e).
\end{align*}

By integral-series identity and by induction on $t$, Theorem \ref{main1-2} follows.

\section{Alternative proof of Theorem \ref{main1}/Proof of Theorem \ref{main2}}
\subsection{Alternative proof of Theorem \ref{main1}}
The derivation relation for MZVs was first proved by Ihara, Kaneko and Zagier \cite{ihara_kaneko_zagier_2006}.
Y.\ Horikawa, K.\ Oyama and the second author \cite{horikawa_oyama_murahara_2017} showed the equivalence of the derivation relation and Theorem \ref{ohno-type_MZVs}. 
\begin{definition}
 For $\k =(k_1,\dots,k_r) \in \Z_{\ge1}^r$, we define Hoffman's dual index of $\k$ by
 \begin{align*}
  \k^{\vee}=(\underbrace{1,\dots,1}_{k_1}+\underbrace{1,\dots,1}_{k_2}+1,\dots,1+\underbrace{1,\dots,1}_{k_r}).
 \end{align*}
\end{definition} 
\begin{theorem}[Horikawa--Oyama--Murahara\ \cite{horikawa_oyama_murahara_2017}] \label{ohno-type_MZVs}
 For $\k \in \Z_{\ge1}^r$ and $l \in \Z_{\ge 0}$, we have
 \begin{align*}
  \sum_{\substack{|\e|=l \\ \e\geq 0}} \zeta^+ (\k+\e)
  =\sum_{\substack{|\e'|=l \\ \e'\geq 0}} \zeta^+ ((\k^\vee+\e')^\vee).
 \end{align*}
\end{theorem}

In this subsection, we prove Theorem \ref{main1} by showing that it is equivalent to Theorem \ref{ohno-type_MZVs}, i.e., we will show the following.  
\begin{theorem}\label{thm3.3}
 Theorem \ref{main1} and Theorem \ref{ohno-type_MZVs} are equivalent.
\end{theorem}
\begin{remark}
 Tanaka \cite{tanaka_2015} shows that the restricted sum formula in \cite{eie_liaw_ong_2009} is written as the linear combination of the derivation relation. 
 Theorem \ref{thm3.3} is a generalization of this result.
\end{remark}

The implications among ``Generalized restricted sum formula'', ``Restricted sum formula'', ``Ohno type relation'' and ``Derivation relation'' for MZVs can be summarized as follows.
\vspace{10pt}

\hspace{3ex} \fbox{\begin{tabular}{c}Generalized restricted sum formula\\(Theorem \ref{main1}, Theorem \ref{main1-2})\end{tabular} } \ {\Large$\supset$} \ \fbox{Restricted sum formula (\cite{eie_liaw_ong_2009})}

\vspace{3pt}

\hspace{18ex} {\Large\rotatebox{90}{$\Leftrightarrow$}} \raisebox{1.2ex}{\small (This paper)}

\vspace{1pt}

\hspace{3ex} \fbox{Ohno type relation (\cite{horikawa_oyama_murahara_2017}, Theorem \ref{ohno-type_MZVs})}

\vspace{3pt}

\hspace{18ex} {\Large\rotatebox{90}{$\Leftrightarrow$}} \raisebox{1.2ex}{\small (\cite{horikawa_oyama_murahara_2017})}

\vspace{1pt}

\hspace{3ex} \fbox{Derivation relation (\cite{ihara_kaneko_zagier_2006})}
\vspace{10pt}
Now we prove Theorem \ref{thm3.3}. The case $r=1$ is obvious. For $r\geq 2$, the following Lemma \ref{lem} gives Theorem \ref{thm3.3}. 
We denote the naive shuffle of two indices $(k_1,\dots, k_r)$ and $(l_1,\dots, l_s)$ by $(k_1,\dots, k_r) \sh (l_1,\dots, l_s)$,
and we extend $\zeta$ linearly. 
For example, $(3,1) \sh (2)=(2,3,1) + (3,2,1) + (3,1,2)$.
For $\k=(k_1,\dots,k_r) \in \Z_{\ge1}^r$ with $r\ge 2$ and $u\in \Z_{\ge0}$, set $\k_{u}:=(k_1,( (k_2,\dots,k_{r-1}) \sh \{1\}^{u} ), k_r)$. 
  We also let
\begin{align*}
 f_L(\k,t):=&(\textrm{L.H.S. of Theorem \ref{main1} for } \k, t), & 
  f_R(\k,t):=&(\textrm{R.H.S. of Theorem \ref{main1} for } \k, t), \\
 g_L(\k,t):=&\sum_{\substack{|\e|=t \\ \e\geq 0}} \zeta^+ (\k+\e), & 
  g_R(\k,t):=&\sum_{\substack{|\e'|=t \\ \e'\geq 0}} \zeta^+ ((\k\,^\vee+\e')^\vee), \\
 f(\k,t):=&f_L(\k,t) - f_R(\k,t), & 
  g(\k,t):=&g_L(\k,t) - g_R(\k,t)
\end{align*}  
and we extend them linearly with respect to the indices.
Under these settings, we have the following:
\begin{lemma} \label{lem}
 For $\k \in \Z_{\ge1}^r$ with $r\geq 2$ and $t\in \Z_{\ge0}$, we have
 \begin{align*}
  f(\k,t)=&-\sum_{u=0}^{t} g(\k_u,t-u), \\
  g(\k,t)=&- \sum_{u=0}^{t}(-1)^u f(\k_u,t-u).   
 \end{align*} 
\end{lemma}

\begin{proof}
To prove the first equation, it is sufficient to show $f_R(\k,t)=\sum_{u=0}^{t} g_L(\k_u,t-u)$ and $f_L(\k,t)=\sum_{u=0}^{t} g_R(\k_u,t-u)$.
 
The proof of the former is obvious as follows:
\begin{align*}
 f_R(\k,t)
 =&\sum_{l=0}^{t}\sum_{\substack{|\e|=l \\ \e\geq 0}} \,\,
   \zeta^+ ( (k_1,( (k_2,\dots,k_{r-1}) \sh \{1\}^{t-l} ), k_r) + \bold{e} ) \\
   =&\sum_{l=0}^{t}\sum_{\substack{|\e|=l \\ \e\geq 0}} \,\,\zeta^+ ( \k_{t-l} + \bold{e} )\\
   =&\sum_{l=0}^{t} \,\,g_L(\k_{t-l},l).
\end{align*}

To prove the latter, we denote  the $m$-fold repeation of ``$1+$''(resp.``$1,$'') by $\boxp^m$(resp. $\boxc^m$), and 1 by $\boxo$. For example, $\zeta (\boxp^3\boxc^2\boxp^0\boxo)=\zeta(1+1+1+1,1,1)=\zeta(4,1,1).$

 \begin{align*}
  g_R(\k_u,t-u)=&\sum_{\substack{|\e'|=t-u \\ \e'\geq 0}} \zeta^+ ((\k_u^\vee+\e')^\vee)\\
  =&\sum_{\substack{|\e'|=t-u \\ \e'\geq 0}} \zeta^+( ((k_1,( (k_2,\dots,k_{r-1}) \sh \{1\}^{u} ), k_r)^\vee +\e')^\vee) \\
  =&\sum_{\substack{\alpha_1+\dots +\alpha_{r-1}=u, \,\alpha_i\geq 0 \\ |\e'|=t-u,\, \e'\geq 0}} \zeta^+( ((k_1, \{1\}^{\alpha_1}, k_2, \{1\}^{\alpha_2}, \dots,k_{r-1},\{1\}^{\alpha_{r-1}},  k_r)^\vee +\e')^\vee) \\
\\
  =&\sum_{\substack{\alpha_1+\dots +\alpha_{r-1}=u, \, \alpha_i\geq 0 \\ |\e'|=t-u,\, \e'\geq 0}} 
  \zeta^+( ((\boxp^{k_1-1} \boxc^{\alpha_1+1} \boxp^{k_2-1} \boxc^{\alpha_2+1} \cdots \\
  & \hspace{5cm} \cdots \boxp^{k_{r-1}-1} \boxc^{\alpha_{r-1}+1} \boxp^{k_r-1} \boxo )^\vee +\e')^\vee )\\
  =&\sum_{\substack{\alpha_1+\dots +\alpha_{r-1}=u, \, \alpha_i\geq 0 \\ |\e'|=t-u,\, \e'\geq 0}} 
  \zeta^+( ((\boxc^{k_1-1} \boxp^{\alpha_1+1} \boxc^{k_2-1} \boxp^{\alpha_2+1} \cdots \\
  & \hspace{5cm} \cdots \boxc^{k_{r-1}-1} \boxp^{\alpha_{r-1}+1} \boxc^{k_r-1} \boxo ) +\e')^\vee) \\
  =&\sum_{\substack{\alpha_1+\dots +\alpha_{r-1}=u \\ e_{1,1}+\dots +e_{r,k_r-1}=t-u \\ \alpha_i\geq0,\, e_{i,j}\geq 0}} 
  \zeta^+((\boxp^{e_{1,1}}\boxc\cdots\boxp^{e_{1,k_1-1}}\boxc \ \boxp^{e_{1,k_1}+1}\\
  &\hspace{3.2cm}\boxp^{\alpha_1} \boxc\boxp^{e_{2,1}}\boxc\cdots\boxp^{e_{2,k_2-2}}\boxc  \boxp^{e_{2,k_2-1}+1} \\
  &\hspace{6cm}\cdots\cdots \\
  &\hspace{3.2cm}\boxp^{\alpha_{r-1}}\boxc\boxp^{e_{r,1}}\boxc\cdots\boxp^{e_{r,k_r-2}}\boxc \ \boxp^{e_{r,k_r-1}} \boxo
)^\vee)\\
%
%
  =&\sum_{\substack{\alpha_1+\dots +\alpha_{r-1}=u \\ e_{1,1}+\dots +e_{r,k_r-1}=t-u \\ \alpha_i\geq0,\,  e_{i,j}\geq 0}} \zeta^+(\boxc^{e_{1,1}}\boxp\cdots\boxc^{e_{1,k_1-1}}\boxp \ 
  \boxc^{e_{1,k_1}+1} \\
  &\hspace{3.2cm}\boxc^{\alpha_1} \boxp\boxc^{e_{2,1}}\boxp\cdots\boxc^{e_{2,k_2-2}}\boxp  \boxc^{e_{2,k_2-1}+1} \\
  &\hspace{6cm}\cdots\cdots \\
  &\hspace{3.2cm}\boxc^{\alpha_{r-1}}\boxp\boxc^{e_{r,1}}\boxp\cdots\boxc^{e_{r,k_r-2}}\boxp \ \boxc^{e_{r,k_r-1}} \boxo
).
 \end{align*} 
 
Take the sum over $u=0, \dots, t$, we have
 \begin{align*}
  \sum^t_{u=0}g_R(\k_u,t-u)
 %
  %
    =&\sum_{\substack{\alpha_1+\dots+\alpha_{r-1}+e_{1,1}+\dots +e_{r,k_r-1}=t \\ \alpha_i\geq 0, e_{i,j}\geq 0}} \zeta^+ (
  \underbrace{\boxc^{e_{1,1}}\boxp\cdots\boxc^{e_{1,k_1-1}}\boxp \  \boxc^{e_{1,k_1}+1}}_{\substack{weight=e_{1,1}+\dots+e_{1,k_1}+k_1 \\ depth=e_{1,1}+\dots+e_{1,k_1}+1}}\\
  &\hspace{2.2cm}\underbrace{\boxc^{\alpha_1}\boxp\boxc^{e_{2,1}}\boxp\cdots\boxc^{e_{2,k_2-2}}\boxp\boxc^{e_{2,k_2-1}+1}}_ {\substack{weight=\alpha_1+e_{2,1}+\dots+e_{2,k_2-1}+k_2 \\ depth=\alpha_1+e_{2,1}+\dots+e_{2,k_2-1}+1}}\\
  & \hspace{4cm}\cdots \cdots \\
  &\hspace{2.2cm}\underbrace{\boxc^{\alpha_{r-1}}\boxp\boxc^{e_{r,1}}\boxp\cdots\boxc^{e_{r,k_r-2}}\boxp \ \boxc^{e_{r,k_r-1}} \boxo}_{\substack{weight=\alpha_{r-1}+e_{r,1}+\dots+e_{r,k_r-1}+k_r \\ depth=\alpha_{r-1}+e_{r,1}+\dots+e_{r,k_r-1}+1}} ) \\
  =&\sum_{\substack{ m_1+\cdots +m_r=r+t \\ m_i\ge1 \, (1\le i\le r) }} \,
   \sum_{\substack{ |\bold{a}_{m_i}|=k_i+m_i-1 \\ (1\le i\le r)} }
  \zeta^+ (\bold{a}_{m_1},\dots ,\bold{a}_{m_r})\\
  =&f_L(\k,t).
 \end{align*} 

 We assume the first equation in the lemma and prove the second by induction on $t$.
 The case $t=0$ is clear. Let $t>0$ and assume $g(\k,t')=- \sum_{u=0}^{t'}(-1)^u f(\k_u,t'-u)$ for all integers $t'$ with $0\leq t'<t$.
 From the first equation, we have
 \begin{align*}
  g(\k,t)=&-f(\k,t)-\sum^t_{u=1}g(\k_u,t-u) \\
  =&-f(\k,t)+\sum^t_{u=1} \sum^{t-u}_{u'=0}(-1)^{u'}f((\k_u)_{u'},t-u-u').
 \end{align*}
 Since $(\k_u)_{u'}=\binom{u+u'}{u} \k_{u+u'}$ and by writing $v=u+u'$, 
 \begin{align*}
  g(\k,t)
  =&-f(\k,t) +\sum^t_{u=1} \sum^{t}_{v=u} (-1)^{v-u} \binom{v}{u} f(\k_{v},t-v)\\
  =&-f(\k,t) +\sum^t_{v=1} (-1)^v \sum^{v}_{u=1} (-1)^{u} \binom{v}{u} f(\k_{v},t-v)\\
  =&-f(\k,t) -\sum^t_{v=1}(-1)^{v} f(\k_{v},t-v)\\
  =&-\sum^t_{u=0}(-1)^{u}f(\k_{u},t-u).
 \end{align*}
\end{proof} 

\subsection{Proof of Theorem \ref{main2}}
The following theorem is called Ohno-type relation for FMZVs. 
This is conjectured by Kaneko \cite{kaneko_2015} and proved by Oyama \cite{oyama_2015}. 
\begin{theorem}[Oyama \cite{oyama_2015}] \label{ohno-type_FMZVs}
 For $\k \in \Z_{\ge1}^r$ and $l \in \Z_{\ge 0}$, we have
 \begin{align*}
  \sum_{\substack{|\e|=l \\ \e\geq 0}} \zeta_{\mathcal{F}} (\k+\e)
  =\sum_{\substack{|\e'|=l \\ \e'\geq 0}}\zeta_{\mathcal{F}} ((\k^\vee+\e')^\vee).
 \end{align*}
\end{theorem}
\begin{remark}
 The derivation relation for FMZVs was conjectured by Oyama and proved by the second author \cite{murahara_2016}. 
 Horikawa, Oyama and the second author \cite{horikawa_oyama_murahara_2017} shows the equivalence of the derivation relation and the above theorem for FMZVs.
\end{remark}

By Theorem \ref{ohno-type_FMZVs}, we can prove Theorem \ref{main2} in exactly the same manner as in the previous subsection.
The relations among ``Generalized restricted sum formula'', ``Ohno type relation'' and ``Derivation relation'' for FMZVs can be summarized as follows.
\vspace{10pt}

\hspace{3ex} \fbox{Generalized restricted sum formula (Theorem \ref{main2})}

\vspace{3pt}

\hspace{18ex} {\Large\rotatebox{90}{$\Leftrightarrow$}} \raisebox{1.2ex}{\small (This paper)}

\vspace{1pt}

\hspace{8ex} \fbox{Ohno type relation (\cite{oyama_2015}, Theorem \ref{ohno-type_FMZVs})}

\vspace{3pt}

\hspace{18ex} {\Large\rotatebox{90}{$\Leftrightarrow$}} \raisebox{1.2ex}{\small (\cite{horikawa_oyama_murahara_2017})}

\vspace{1pt}

\hspace{8ex} \fbox{Derivation relation (\cite{murahara_2016})}

\vspace{10pt}

\section*{Acknowledgements}
The authors would like to thank Professor Masanobu Kaneko for valuable comments and suggestions.

\end{document}